\documentclass[oneside,11pt]{amsart}

\makeatletter
\@namedef{subjclassname@2020}{
  \textup{2020} Mathematics Subject Classification}
\makeatother
\usepackage[a4paper,width=140mm,top=27mm,bottom=27mm]{geometry}
\usepackage{cases,xcolor}
\usepackage{amsmath,amssymb}
\usepackage[hidelinks]{hyperref}

\newtheorem{theorem}{Theorem}[section]
\newtheorem{lemma}[theorem]{Lemma}
\newtheorem{definition}[theorem]{Definition}

\theoremstyle{definition}
\newtheorem{remx}[theorem]{Remark}
\newenvironment{remark}
  {\pushQED{\qed}\remx}
  {\popQED\endremx}

\newcommand{\N}{\mathbb{N}}
\newcommand{\R}{\mathbb{R}}

\newcommand{\T}{\mathbb{T}}

\newcommand{\pt}{\partial}

\newcommand{\Z}{{\mathbb{Z}}}

\numberwithin{equation}{section}

\begin{document}

\address{Yongming Luo
\newline \indent
Faculty of Computational Mathematics and Cybernetics
\newline \indent Shenzhen MSU-BIT University, China}
\email{luo.yongming@smbu.edu.cn}

\title[Dancer-type solutions via semivirial-vanishing geometry]{On Dancer-type solutions for the Lane--Emden equation via semivirial-vanishing geometry}
\author{Yongming Luo}

\begin{abstract}
Aubin--Talenti bubbles describe the decaying positive solutions of the
zero-frequency critical Lane--Emden equation in Euclidean space. By appealing to bifurcation methods, Dancer
constructed in his seminar paper \cite{DancerSolution} positive-frequency solutions to the Lane--Emden equation which decay in
the noncompact directions and are periodic in one direction.  Alternatively, we
give in this paper an energy-based variational construction of such Dancer-type solutions via
the semivirial-vanishing geometry developed in author's recent work for studying focusing NLS on waveguide manifolds.  The main new ingredient is a strict
sub-bubbling estimate below the Euclidean Sobolev threshold.  Unlike the usual
Brezis--Nirenberg mechanism, no lower-order focusing perturbation is available
in our model.  Instead, the energy drop is produced by the bounded periodic
direction: truncating a Euclidean bubble to one period removes a leading-order
part of the gradient tail, while the nonlinear tail is of lower order.  This
restores compactness of minimizing sequences and yields normalized ground
states for every prescribed mass, thereby answering an open question from
\cite{Luo_energy_crit}.
\end{abstract}


\maketitle

\section{Introduction}
Let
\[
        D=d+1,
        \qquad
        2^*=\frac{2D}{D-2}=2+\frac4{d-1},
        \qquad d\ge2.
\]
The classical energy-critical Lane--Emden equation on $\R^D$ is given by
\begin{equation}
\label{eq:critical-LE-euclidean}
        -\Delta U=U^{2^*-1}.
\end{equation}
Its finite-energy positive solutions are precisely the Aubin--Talenti bubbles, up to translations and dilations; this is the content of the sharp Sobolev theory of Talenti and the classification theory for positive critical solutions \cite{Talenti1976,Liouville3}.  Thus, in the fully Euclidean setting, the zero-frequency equation \eqref{eq:critical-LE-euclidean} has a rigid family of decaying solutions.

A remarkable phenomenon discovered by Dancer \cite{DancerSolution} is that this rigidity disappears once one imposes a periodic structure in one direction.  By a bifurcation argument, Dancer showed that semilinear elliptic equations on Euclidean space can possess positive solutions which decay in all but one direction and are periodic in the remaining direction. In the critical setting considered here, such solutions are naturally modeled by the positive-frequency Lane--Emden equation on \(\mathbb R^d\times\mathbb T\), namely, we consider positive solutions on the waveguide manifold
\[
        \mathbb X:=\R^d_x\times\T_y,
        \qquad \T=\R/2\pi\Z,
\]
of the positive-frequency Lane--Emden equation
\begin{equation}
\label{nls}
        -\Delta_{x,y}u+\beta u=u^{2^*-1},
        \qquad \beta>0.
\end{equation}
We shall refer to such solutions as Dancer-type solutions.  The aim of this paper is to construct such solutions by a variational energy method and, at the same time, to obtain quantitative information which is not visible from the bifurcation construction.

The waveguide geometry $\R^d\times\T$ is also a natural setting for nonlinear Schr\"odinger equations arising in nonlinear optics and related physical models; see, for instance, \cite{waveguide_ref_1,waveguide_ref_2,waveguide_ref_3}.  From the mathematical point of view, dispersive equations on product spaces have been studied extensively in the last decade.  Among all, we mention the representative work of Terracini--Tzvetkov--Visciglia on ground states on product spaces \cite{TTVproduct2014}, the scattering results of Tzvetkov--Visciglia \cite{TNCommPDE,TzvetkovVisciglia2016}, and the energy-critical scattering theory of Hani--Pausader on $\R\times\T^2$ \cite{HaniPausader}.  In the focusing case, a series of works of the author introduced and developed the {\it{semivirial-vanishing geometry}} as a variational mechanism for detecting sharp scattering/blow-up thresholds and normalized ground states on waveguides \cite{luo2021sharp,Luo_inter,Luo_energy_crit}.  We also mention subsequent related developments in critical and combined-power waveguide problems \cite{RmT1,R1T1Scattering,ZhaoZheng2021,LuoSIMA,ForcellaLuoZhao,LuoCriticalScattering}.

The starting point of the semivirial approach is the observation that, on $\R^d\times\T$, the virial identity relevant to the dispersive directions involves only the $x$-gradient.  This leads to the semivirial functional
\[
        Q(u):=\|\nabla_xu\|_2^2-\frac dD\|u\|_{2^*}^{2^*}.
\]
For a prescribed mass $c>0$, we study the constrained minimization problem
\begin{equation}
\label{mc first def}
        m_c:=\inf\{E(u):u\in H^1(\R^d\times\T),\ M(u)=c,\ Q(u)=0\},
\end{equation}
where
\[
        M(u):=\|u\|_2^2,
        \qquad
        E(u):=\frac12\|\nabla_{x,y}u\|_2^2-\frac{D-2}{2D}\|u\|_{2^*}^{2^*}.
\]
The Euler--Lagrange equation associated with an optimizer of \eqref{mc first def} is precisely \eqref{nls} (in fact, it can be shown that $Q(u)$ is a natural constraint, see Lemma \ref{minimizer is solution}).  In the small-mass regime, the strict comparison
proved in \cite{Luo_energy_crit} forces these optimizers to depend nontrivially
on the periodic variable; hence they provide Dancer-type partially periodic
solutions. More precisely, we have the following $y$-dependence result from \cite{Luo_energy_crit}.

\begin{theorem}[$y$-dependence of the ground states, \cite{Luo_energy_crit}]\label{thm threshold mass}
There exists some $c_*\in(0,\infty)$ such that for any $c\in(0,c_*)$, any optimizer $u_c$ of $m_c$ must satisfy $\pt_y u_c\neq 0$.
\end{theorem}

Theorem \ref{thm threshold mass} shows that any small-mass optimizer, if it exists, must depend nontrivially on the periodic variable.  However, the existence of such optimizers for $0<c<c_*$ was left open in \cite{Luo_energy_crit}.  The main purpose of the present paper is to resolve this open problem and to show that, in fact, optimizers exist for every prescribed mass.

Our main result is the following.

\begin{theorem}\label{main theorem}
For any mass $c\in (0,\infty)$ the variational problem $m_c$ defined by \eqref{mc first def} admits an optimizer. Moreover, $u_c$ may be chosen strictly positive and it solves \eqref{nls} for some $\beta=\beta_c>0$.
\end{theorem}

The proof of Theorem~\ref{main theorem} is based on a strict sub-bubbling estimate.  Let $\mathcal S$ denote the sharp Sobolev constant on $\R^D$, namely
\[
        \mathcal S
        :=\inf_{0\ne f\in \dot H^1(\R^D)}
        \frac{\|\nabla f\|_{L^2(\R^D)}^2}
        {\|f\|_{L^{2^*}(\R^D)}^2}.
\]
The Euclidean Aubin--Talenti bubble energy is
\[
        m_{\rm bub}:=\frac{\mathcal S^{D/2}}{D}.
\]
The key theorem is the strict inequality below.

\begin{theorem}\label{theorem strict bound}
For any $c\in(0,\infty)$ we have $m_c<m_{\rm bub}$.
\end{theorem}

Theorem~\ref{theorem strict bound} is the main new variational input.  It should be compared with the non-strict estimate
\[
        m_c\le \frac{\mathcal S^{D/2}}{D}
\]
proved in \cite[Lem. 3.6]{Luo_energy_crit}.  In a standard Brezis--Nirenberg type argument, a strict inequality below the Euclidean bubble level is often obtained from a lower-order focusing perturbation \cite{brezis_nirenberg}.  Such a mechanism is absent here: the problem contains only the critical nonlinearity, and there is no additional lower-order nonlinear term whose sign can lower the energy.  

The key observation of this paper is that the missing Brezis--Nirenberg input can be replaced by a purely geometric mechanism.  The strict energy decrease is
produced by the bounded periodic direction itself. More precisely, we place a rescaled Aubin--Talenti bubble in one fundamental period of the waveguide.  In rescaled coordinates, this amounts to keeping the Euclidean bubble in a slab $\R^d\times(-R,R)$ and discarding the tails $|Y|>R$.  The discarded tail contributes to the gradient energy at order $R^{2-D}$, whereas the nonlinear tail contributes only at order $R^{-D}$.  Consequently, after a careful projection back to the semivirial constraint $Q=0$, the energy remains strictly below $m_{\rm bub}$.  This is the bounded-direction analogue of the Brezis--Nirenberg test-function mechanism, but the sign of the correction comes from the loss of Euclidean tail energy caused by the compact direction.

Once Theorem~\ref{theorem strict bound} is available, the proof of compactness follows the usual critical strategy of concentration compactness.  A minimizing sequence for $m_c$ can lose compactness either by a scale-one translation along $\R^d$ or by forming a Euclidean critical bubble.  The first alternative gives a nonzero weak limit and hence compactness.  The second alternative carries at least the Euclidean bubble energy $m_{\rm bub}$; the strict estimate of Theorem~\ref{theorem strict bound} rules it out.  The remaining Brezis--Lieb splitting and Lagrange multiplier argument then yield an optimizer of $m_c$.  In this way, the semivirial-vanishing geometry gives an alternative energetic construction of Dancer-type solutions to the Lane--Emden equation.

Combining the present result with the intercritical theory developed in \cite{Luo_inter} and the energy-critical framework of \cite{Luo_energy_crit}, one obtains a unified variational approach to Dancer-type partially periodic solutions from the mass-supercritical regime up to the energy-critical endpoint.  The contribution of the present paper is precisely to close the endpoint compactness gap left open in \cite{Luo_energy_crit}.

The paper is organized as follows.  Section~\ref{sec:notation} fixes the notation and definitions that will be used throughout the paper.  Section~\ref{sec useful lemma} recalls the elementary properties of the semivirial constraint.  Section~\ref{sec:profile} states the static critical profile decomposition on $\R^d\times\T$ in the form used later.  Section~3 proves the strict sub-bubbling bound, Theorem~\ref{theorem strict bound}.  Section~4 proves the compactness of minimizing sequences and completes the proof of Theorem~\ref{main theorem}.

\section{Preliminaries}

\subsection{Notation and definitions}\label{sec:notation}
We use the notation $A\lesssim B$ if there exists a constant $C>0$ such that $A\le CB$.  Similarly, $A\gtrsim B$ means $B\lesssim A$, and $A\sim B$ means both $A\lesssim B$ and $B\lesssim A$. 

We write
\[
        \mathbb X:=\R^d_x\times\T_y,
        \qquad
        D:=d+1,
        \qquad
        2^*:=\frac{2D}{D-2}=\frac{2(d+1)}{d-1}.
\]
For $q\in(0,\infty)$, the norm $\|\cdot\|_q$ will always denote the $L^q(\mathbb X)$ norm, unless the underlying domain is explicitly indicated.  We also use $H^1_{x,y}:=H^1(\mathbb X)$ and $H^1_x:=H^1(\R^d)$.  For functions on $\R^{D}$, the full gradient is denoted by $\nabla_{\R^D}$ or simply $\nabla$ when no confusion is possible.

In the following we shall record the semivirial quantities and scaling identities used in \cite{Luo_energy_crit}. For $u\in H^1_{x,y}$, define the energies on a waveguide manifold:
\begin{align*}
        M(u)&:=\|u\|_2^2,
        \\
        E(u)&:=\frac12\|\nabla_{x,y}u\|_2^2-\frac{D-2}{2D}\|u\|_{2^*}^{2^*},
        \\
        Q(u)&:=\|\nabla_xu\|_2^2-\frac dD\|u\|_{2^*}^{2^*},
        \\
        I(u)&:=E(u)-\frac12Q(u)
        =\frac12\|\partial_yu\|_2^2+\frac1{2D}\|u\|_{2^*}^{2^*}.
\end{align*}
For $c>0$ set
\begin{align*}
        S(c):=\{u\in H^1_{x,y}:M(u)=c\}
        ,\quad
        V(c):=\{u\in S(c):Q(u)=0\}.
\end{align*}
The variational problem $m_c$ is then defined by
\begin{align*}
        m_c&:=\inf\{E(u):u\in V(c)\}.
\end{align*}
For $u\in H^1_{x,y}$ and $t>0$ define the mass-preserving $x$-scaling
\begin{equation*}
        u_t(x,y):=t^{d/2}u(tx,y).
\end{equation*}
A direct computation gives
\begin{align*}
        M(u_t)&=M(u),
        \\
        \|\nabla_xu_t\|_2^2&=t^2\|\nabla_xu\|_2^2,
        \\
        \|\partial_yu_t\|_2^2&=\|\partial_yu\|_2^2,
        \\
        \|u_t\|_{2^*}^{2^*}&=t^\sigma\|u\|_{2^*}^{2^*},
\end{align*}
with $ \sigma:=d\left(\frac{2^*}{2}-1\right)=\frac{2d}{d-1}>2$. Consequently,
\begin{align}
        Q(u_t)&=t^2\|\nabla_xu\|_2^2-\frac dD t^\sigma\|u\|_{2^*}^{2^*},\notag
        \\
        \frac{d}{dt}E(u_t)&=t^{-1}Q(u_t).
        \label{scale:Ederivative}
\end{align}

\subsection{Some useful lemmas}\label{sec useful lemma}

We record several elementary consequences that will be used repeatedly. For their proofs, we refer to \cite{Luo_inter}.

\begin{lemma}[Projection onto $Q=0$]\label{lem:projection}
Let $u\in H^1_{x,y}$ be nonzero.
\begin{enumerate}
\item If $Q(u)<0$, then there exists a unique $t_u\in(0,1)$ such that $Q(u_{t_u})=0$.
\item If $Q(u)=0$, then $E(u)=I(u)$.
\item If $Q(u)<0$ and $t_u$ is as in (1), then
\[
        I(u_{t_u})<I(u).
\]
\end{enumerate}
\end{lemma}

\begin{lemma}[An equivalent useful characterization of $m_c$]\label{lem:lecoz}
For every $c>0$,
\begin{equation*}
        m_c=\tilde m_c:=\inf\{I(u):u\in S(c),\ Q(u)\le0\}.
\end{equation*}
\end{lemma}

\begin{lemma}[Monotonicity of $c\mapsto m_c$]\label{lem:monotonicity}
The mapping $c\mapsto m_c$ is nonincreasing on $(0,\infty)$.
\end{lemma}

\begin{lemma}[Scale-invariant Gagliardo-Nirenberg inequality on $\R^d\times\T$]\label{lemma gn additive}
There exists some $C>0$ such that for all $u\in H_{x,y}^1$ we have
\begin{align*}
\|u\|_{2^*}^{2^*}
\leq C\|\nabla_x u\|_2^{\frac{2d}{d-1}}
\left(\|u\|_{2}^{\frac{2}{d-1}}+\|\pt_y u\|_{2}^{\frac{2}{d-1}}\right).
\end{align*}
\end{lemma}

\begin{lemma}[Characterization of a minimizer as a standing wave solution]\label{minimizer is solution}
For any $c\in(0,\infty)$ an optimizer $u$ of $m_c$ is a solution of
\begin{align*}
-\Delta_{x,y}u+\beta u=|u|^{\frac{4}{d-1}}u
\end{align*}
with some $\beta\in\R$.
\end{lemma}

\subsection{Profile decomposition}\label{sec:profile}
We recall the static part of the critical profile decomposition on
\(\mathbb X=\R^d\times\T\).  The decomposition used below is the time-zero version of
the waveguide profile decomposition developed for energy-critical problems on
product spaces introduced by Hani and Pausader \cite{HaniPausader}, which in the semivirial
setting was further applied in \cite{Luo_energy_crit}.  Since our compactness argument is
purely variational, no time translations or nonlinear profiles are needed.  We
nevertheless keep the terminology of frames and profiles, because it is important
to distinguish scale-one profiles from Euclidean concentrating profiles.

Let
\[
        d_{\mathbb X}\bigl((x,y),(x',y')\bigr)
        :=|x-x'|+\operatorname{dist}_{\T}(y,y')
\]
be the product distance.  For \(z_0=(x_0,y_0)\in\mathbb X\), define the spatial
translation symmetry
\begin{equation*}
        \Pi_{z_0}f(x,y):=f(x-x_0,y-y_0).
\end{equation*}
Here subtraction in the second variable is understood modulo \(2\pi\).

The first object needed for the decomposition is the operator which inserts a whole-space profile into a small coordinate patch of the waveguide.  This is the mechanism by which loss of compactness at the critical Sobolev exponent is recorded as an \(\R^D\)-bubble.  We recall it before introducing frames.  Fix
\(\eta\in C_c^\infty(\R^D;[0,1])\) such that \(\eta=1\) on a neighbourhood of the
origin.  For \(\phi\in\dot H^1(\R^D)\) and \(N\ge1\), set
\begin{equation}
\label{eq:phiN}
        \phi_N(Z):=N^{\frac{D-2}{2}}\eta(N^{1/2}Z)\phi(NZ),
        \qquad Z=(X,Y)\in\R^d\times\R.
\end{equation}
Identifying \(\T\) with \([-\pi,\pi]\), we define \(T_N\phi\) on \(\R^d\times\T\) by
periodically extending \(\phi_N\) in the \(Y\)-variable.  Equivalently, after
identifying \((x,y)\in\R^d\times\T\) with \((x,y)\in\R^d\times[-\pi,\pi]\), we set
\(T_N\phi(x,y)=\phi_N(x,y)\) on the support of \(\phi_N\) and \(T_N\phi=0\)
outside the corresponding coordinate patch.  For large \(N\), the support of
\(\phi_N\) is contained in a ball of radius \(O(N^{-1/2})\), hence this definition is
unambiguous.

\begin{lemma}[Euclidean approximation, \cite{Luo_energy_crit}]\label{lem:euclidean-operator}
Let \(\phi\in\dot H^1(\R^D)\), and let \(T_N\phi\) be defined by
\eqref{eq:phiN}.  Then, as \(N\to\infty\),
\begin{align}
        \|T_N\phi\|_{L^2(\mathbb X)}^2&=o_N(1),\notag\\
        \|\nabla_xT_N\phi\|_{L^2(\mathbb X)}^2
        &=\|\nabla_X\phi\|_{L^2(\R^D)}^2+o_N(1),\label{eq:TNx}\\
        \|\partial_yT_N\phi\|_{L^2(\mathbb X)}^2
        &=\|\partial_Y\phi\|_{L^2(\R^D)}^2+o_N(1),\label{eq:TNy}\\
        \|T_N\phi\|_{L^{2^*}(\mathbb X)}^{2^*}
        &=\|\phi\|_{L^{2^*}(\R^D)}^{2^*}+o_N(1).\notag
\end{align}
Consequently,
\begin{equation*}
        \|\nabla_{x,y}T_N\phi\|_{L^2(\mathbb X)}^2
        =\|\nabla_{\R^D}\phi\|_{L^2(\R^D)}^2+o_N(1).
\end{equation*}
\end{lemma}

\begin{remark}
Lemma~\ref{lem:euclidean-operator} is precisely the static asymptotic statement
used in \cite[Lem. 3.14]{Luo_energy_crit}, with the ambient dimension \(4\)
replaced by \(D=d+1\) and the critical exponent \(2^*=2D/(D-2)\).  The separated
convergences \eqref{eq:TNx} and \eqref{eq:TNy} follow by applying the same
argument componentwise to the \(X\)- and \(Y\)-derivatives.  We shall use this
lemma as an input and will not repeat its proof.
\end{remark}

The next definition records the two possible ways in which compactness can fail for a bounded \(H^1\)-sequence: ordinary translation in the unbounded \(x\)-direction, and concentration at scales much smaller than the period of \(\T\).  These alternatives are encoded by static frames.

\begin{definition}[Static frames and profile operators]\label{def:static-frame}
A static frame is a sequence
\[
        \mathcal F=(N_n,z_n)_{n\ge1},
        \qquad N_n\ge1,
        \qquad z_n=(x_n,y_n)\in\mathbb X.
\]
There are two types of frames.
\begin{enumerate}
\item A \emph{\underline{scale-one}} frame satisfies \(N_n\equiv1\).  It acts on a profile
\(\phi\in H^1(\mathbb X)\) by
\begin{equation*}
        P_{\mathcal F,n}\phi:=\Pi_{z_n}\phi.
\end{equation*}
\item A \emph{\underline{Euclidean}} frame satisfies \(N_n\to\infty\).  It acts on a profile
\(\phi\in\dot H^1(\R^D)\) by
\begin{equation*}
        P_{\mathcal F,n}\phi:=\Pi_{z_n}T_{N_n}\phi.
\end{equation*}
\end{enumerate}
Two frames \(\mathcal F^j=(N_n^j,z_n^j)\) and \(\mathcal F^k=(N_n^k,z_n^k)\) are
called orthogonal if
\begin{equation}
\label{eq:frame-orthogonality}
        \left|\log\frac{N_n^j}{N_n^k}\right|
        +N_n^j d_{\mathbb X}(z_n^j,z_n^k)
        \longrightarrow\infty
        \qquad (n\to\infty).
\end{equation}
\end{definition}

\begin{remark}
When both frames are scale-one, \eqref{eq:frame-orthogonality} reduces to \(|x_n^j-x_n^k|\to\infty\), after
passing to a subsequence in the compact \(y\)-variable.  If one frame is scale-one
and the other Euclidean, orthogonality follows automatically from the logarithmic
scale separation.
\end{remark}

We now state the static decomposition in the form needed later.  Its role is to separate a minimizing sequence into scale-one objects, Euclidean bubbles, and a remainder which is small in the critical Lebesgue norm.  The Pythagorean expansions below are the quantitative input used in the compactness proof in Section~\ref{sec:mainproof}.

\begin{lemma}[Static critical profile decomposition, \cite{HaniPausader}]\label{lem:profile}
Let \((f_n)_n\) be a bounded sequence in \(H^1(\mathbb X)\).  After passing to a
subsequence, there exist a number \(J_*\in\N\cup\{\infty\}\), nonzero profiles
\((\phi^j)_{1\le j<J_*}\), pairwise orthogonal static frames
\((\mathcal F^j)_{1\le j<J_*}\), and remainders \((r_n^J)_{n}\subset H^1(\mathbb X)\)
such that, for each finite \(J<J_*\),
\begin{equation*}
        f_n=\sum_{j=1}^J P_{\mathcal F^j,n}\phi^j+r_n^J.
\end{equation*}
Each \(\phi^j\) belongs either to \(H^1(\mathbb X)\), in which case
\(\mathcal F^j\) is a scale-one frame, or to \(\dot H^1(\R^D)\), in which case
\(\mathcal F^j\) is a Euclidean frame.  Moreover, for every finite \(J<J_*\), one has
\begin{align}
        \|f_n\|_2^2
        &=\sum_{j=1}^J\|P_{\mathcal F^j,n}\phi^j\|_2^2+
          \|r_n^J\|_2^2+o_n(1),\label{eq:prof-mass}\\
        \|\nabla_x f_n\|_2^2
        &=\sum_{j=1}^J\|\nabla_xP_{\mathcal F^j,n}\phi^j\|_2^2+
          \|\nabla_xr_n^J\|_2^2+o_n(1),\label{eq:prof-x}\\
        \|\partial_y f_n\|_2^2
        &=\sum_{j=1}^J\|\partial_yP_{\mathcal F^j,n}\phi^j\|_2^2+
          \|\partial_yr_n^J\|_2^2+o_n(1),\notag\\
        \|f_n\|_{2^*}^{2^*}
        &=\sum_{j=1}^J\|P_{\mathcal F^j,n}\phi^j\|_{2^*}^{2^*}+
          \|r_n^J\|_{2^*}^{2^*}+o_n(1),\label{eq:prof-lp}
\end{align}
and the remainder is small in the critical Lebesgue norm:
\begin{equation}
\label{eq:prof-rem}
        \lim_{J\to J_*}\limsup_{n\to\infty}\|r_n^J\|_{2^*}=0.
\end{equation}
\end{lemma}

\begin{remark}
For Euclidean profiles, the quantities in \eqref{eq:prof-mass}--\eqref{eq:prof-lp}
are identified by Lemma~\ref{lem:euclidean-operator}.  In particular,
\begin{align*}
        \|P_{\mathcal F^j,n}\phi^j\|_2^2&=o_n(1),\\
        \|\nabla_xP_{\mathcal F^j,n}\phi^j\|_2^2
        &=\|\nabla_X\phi^j\|_{L^2(\R^D)}^2+o_n(1),\\
        \|\partial_yP_{\mathcal F^j,n}\phi^j\|_2^2
        &=\|\partial_Y\phi^j\|_{L^2(\R^D)}^2+o_n(1),\\
        \|P_{\mathcal F^j,n}\phi^j\|_{2^*}^{2^*}
        &=\|\phi^j\|_{L^{2^*}(\R^D)}^{2^*}+o_n(1).
\end{align*}
\end{remark}

\section{Strict upper bound for the ground state energy}\label{sec:strict}
This section proves Theorem~\ref{theorem strict bound}.  The proof is an Aubin--Talenti test-function argument.  The main point is that the loss of the Euclidean bubble tail in the gradient part is of order $R^{2-D}$, whereas the loss of the nonlinear term is only of order $R^{-D}$.

The proof uses a family of trial functions obtained from a Euclidean Sobolev optimizer.  The following normalization fixes the constants appearing in the comparison with the sharp Sobolev bubble.

Let $W$ be an Aubin--Talenti optimizer for the sharp Sobolev inequality on $\R^D$, chosen so that
\begin{equation}
\label{eq:W}
        W(Z)=\alpha_D(1+|Z|^2)^{-\frac{D-2}{2}},
        \qquad Z=(X,Y)\in\R^d\times\R,
\end{equation}
and
\begin{equation*}
        -\Delta W=W^{2^*-1}.
\end{equation*}
By the well-known equality case in the sharp Sobolev inequality,
\begin{equation*}
        K:=\int_{\R^D}|\nabla W|^2
        =\int_{\R^D}W^{2^*}
        =\mathcal S^{D/2}.
\end{equation*}
Since $W$ is radial,
\begin{equation}
\label{eq:radial-split}
        \int_{\R^D}|\nabla_XW|^2=\frac dD K,
        \qquad
        \int_{\R^D}|\partial_YW|^2=\frac1D K.
\end{equation}

Fix $\chi\in C_c^\infty([0,\infty);[0,1])$ with $\chi=1$ on $[0,1]$ and $\chi=0$ on $[2,\infty)$.  For $R\gg1$ set
\[
        \lambda_R:=\frac R\pi,
        \qquad
        L_R:=R^2,
        \qquad
        \eta_R(X):=\chi\left(\frac{|X|}{L_R}\right).
\]
For $y\in[-\pi,\pi]$ define
\begin{equation*}
        U_R(x,y):=
        \lambda_R^{\frac{D-2}{2}}
        \eta_R(\lambda_Rx)W(\lambda_Rx,\lambda_Ry).
\end{equation*}
Since $W(X,Y)$ is even in $Y$, $U_R(x,\pi)=U_R(x,-\pi)$.  Hence $U_R$ defines an element of $H^1(\mathbb X)$.  The $X$-cutoff is not needed in every dimension (recall that the Aubin--Talenti bubbles on $\R^D$ belong to $L^2$ in the case $D\geq 5$), but it makes $L^2$-admissibility uniform and its energy error is of lower order.

We next compute the three quantities which enter the semivirial and the energy of \(U_R\).  After the change of variables below, the waveguide manifold corresponds to a slab \(|Y|<R\) in \(\R^D\).  The negative correction to the energy will come from the Euclidean tail removed outside this slab.

Let
\[
        \Sigma_R:=\R^d_X\times(-R,R)_Y.
\]
Changing variables $X=\lambda_Rx$, $Y=\lambda_Ry$, define
\begin{align*}
        A_R&:=\|\nabla_xU_R\|_2^2
        =\int_{\Sigma_R}|\nabla_X(\eta_RW)|^2\,dX\,dY,\\
        B_R&:=\|\partial_yU_R\|_2^2
        =\int_{\Sigma_R}\eta_R^2|\partial_YW|^2\,dX\,dY,\\
        C_R&:=\|U_R\|_{2^*}^{2^*}
        =\int_{\Sigma_R}\eta_R^{2^*}W^{2^*}\,dX\,dY.
\end{align*}
Define the $Y$-tails
\begin{align*}
        a_R&:=\int_{\{|Y|>R\}}|\nabla_XW|^2\,dX\,dY,\\
        b_R&:=\int_{\{|Y|>R\}}|\partial_YW|^2\,dX\,dY,\\
        \gamma_R&:=\int_{\{|Y|>R\}}W^{2^*}\,dX\,dY.
\end{align*}

The following elementary estimate identifies the main-order size of the missing tail.  It is important that the gradient tail is of order \(R^{2-D}\), while the nonlinear tail is only \(O(R^{-D})\).

\begin{lemma}[Tail asymptotics]\label{lem:tail}
There exist constants $c_a,c_b>0$ such that
\[
        a_R\sim c_aR^{2-D},
        \qquad
        b_R\sim c_bR^{2-D},
        \qquad
        \gamma_R=O(R^{-D})=o(R^{2-D}).
\]
\end{lemma}
\begin{proof}
Write $r=(|X|^2+Y^2)^{1/2}$.  From \eqref{eq:W},
\[
        W'(r)=-\alpha_D(D-2)r(1+r^2)^{-D/2}.
\]
Thus
\[
        \nabla_XW=W'(r)\frac Xr,
        \qquad
        \partial_YW=W'(r)\frac Yr
\]
and therefore
\begin{align*}
        |\nabla_XW|^2
        &=\alpha_D^2(D-2)^2\frac{|X|^2}{(1+|X|^2+Y^2)^D},\\
        |\partial_YW|^2
        &=\alpha_D^2(D-2)^2\frac{Y^2}{(1+|X|^2+Y^2)^D},\\
        W^{2^*}&=\alpha_D^{2^*}(1+|X|^2+Y^2)^{-D}.
\end{align*}
For $a_R$, set $X=R\xi$ and $Y=R\zeta$.  Then
\[
        a_R=\alpha_D^2(D-2)^2R^{2-D}
        \int_{\{|\zeta|>1\}}
        \frac{|\xi|^2}{(R^{-2}+|\xi|^2+\zeta^2)^D}
        \,d\xi\,d\zeta.
\]
Dominated convergence gives
\[
        R^{D-2}a_R\to
        \alpha_D^2(D-2)^2
        \int_{\{|\zeta|>1\}}
        \frac{|\xi|^2}{(|\xi|^2+\zeta^2)^D}
        \,d\xi\,d\zeta=:c_a>0.
\]
The integral is finite since the integrand is $O(|(\xi,\zeta)|^{-2D+2})$ at infinity, and $2D-2>D$ for $D\ge3$.  The proof for $b_R$ is identical and gives a positive constant
\[
        c_b:=\alpha_D^2(D-2)^2
        \int_{\{|\zeta|>1\}}
        \frac{\zeta^2}{(|\xi|^2+\zeta^2)^D}
        \,d\xi\,d\zeta>0.
\]
Finally,
\[
        \gamma_R
        =\alpha_D^{2^*}R^{-D}
        \int_{\{|\zeta|>1\}}
        (R^{-2}+|\xi|^2+\zeta^2)^{-D}
        \,d\xi\,d\zeta
        =O(R^{-D}).
\]
Since $D\ge3$, $R^{-D}=o(R^{2-D})$.
\end{proof}

The \(X\)-cutoff is inserted only to guarantee uniform admissibility in low dimensions.  The next lemma shows that it does not alter the leading-order slab-tail expansion.

\begin{lemma}[Effect of the $X$-cutoff]\label{lem:cutoff}
As $R\to\infty$,
\begin{align}
        A_R&=\frac dD K-a_R+o(R^{2-D}),\label{eq:Aexp}\\
        B_R&=\frac1D K-b_R+o(R^{2-D}),\notag\\
        C_R&=K-\gamma_R+o(R^{2-D}).\label{eq:Cexp}
\end{align}
\end{lemma}
\begin{proof}
The claim for $B_R$ and $C_R$ follows once we show that the part removed by $\eta_R$ in the region $|X|\gtrsim L_R$ is $o(R^{2-D})$.  On $|Y|<R$ and $|X|\gtrsim L_R$ we have
\[
        |\nabla W|^2\lesssim |X|^{-2D+2},
        \qquad
        W^2\lesssim |X|^{-2D+4},
        \qquad
        W^{2^*}\lesssim |X|^{-2D}.
\]
Therefore
\[
        \int_{|Y|<R,\ |X|>L_R}|\nabla W|^2
        \lesssim
        R\int_{L_R}^\infty \rho^{d-1}\rho^{-2D+2}\,d\rho.
\]
Since $d=D-1$, the right-hand side is
\[
        \lesssim R L_R^{1-D}=R^{3-2D}=o(R^{2-D}).
\]
Likewise,
\[
        \int_{|Y|<R,\ |X|>L_R}W^{2^*}
        \lesssim
        R\int_{L_R}^\infty \rho^{d-1}\rho^{-2D}\,d\rho
        \lesssim R L_R^{-D}=o(R^{2-D}).
\]
It remains to handle the derivative of the cutoff in $A_R$.  Since $|\nabla_X\eta_R|\lesssim L_R^{-1}$ and $\operatorname{supp}\nabla_X\eta_R\subset\{L_R\le |X|\le2L_R\}$,
\begin{align*}
        \int_{|Y|<R}|\nabla_X\eta_R|^2W^2
        &\lesssim L_R^{-2}R\int_{L_R}^{2L_R}\rho^{d-1}\rho^{-2D+4}\,d\rho \\
        &\lesssim R L_R^{1-D}=o(R^{2-D}).
\end{align*}
The cross term is controlled by Cauchy--Schwarz:
\begin{align*}
&\left|2\int_{\Sigma_R}\eta_RW\nabla_X\eta_R\cdot\nabla_XW\right|  \\
&\qquad\le
2\left(\int_{\operatorname{supp}\nabla\eta_R}|\nabla W|^2\right)^{1/2}
\left(\int_{\operatorname{supp}\nabla\eta_R}|\nabla\eta_R|^2W^2\right)^{1/2}
=o(R^{2-D}).
\end{align*}
Combining these estimates with \eqref{eq:radial-split} and the definitions of $a_R,b_R,\gamma_R$ gives \eqref{eq:Aexp}--\eqref{eq:Cexp}.
\end{proof}

Having all the preliminaries, we are now ready for giving the proof of Theorem \ref{theorem strict bound}.

\begin{proof}[Proof of Theorem~\ref{theorem strict bound}]
By Lemmas~\ref{lem:tail} and \ref{lem:cutoff},
\begin{align*}
        Q(U_R)
        &=A_R-\frac dD C_R \\
        &=-a_R+\frac dD\gamma_R+o(R^{2-D})<0
\end{align*}
for all sufficiently large $R$.  Moreover,
\begin{align*}
        E(U_R)
        &=\frac12(A_R+B_R)-\frac{D-2}{2D}C_R \\
        &=\frac KD-\frac12(a_R+b_R)+\frac{D-2}{2D}\gamma_R+o(R^{2-D}) \\
        &\le \frac KD-\kappa R^{2-D}
\end{align*}
for some $\kappa>0$ and all large $R$.  Since $K=\mathcal S^{D/2}$, this reads
\begin{equation}
\label{eq:energy-before-proj}
        E(U_R)\le \frac{\mathcal S^{D/2}}D-\kappa R^{2-D}.
\end{equation}

By Lemma~\ref{lem:projection}, there exists $t_R\in(0,1)$ such that $Q((U_R)_{t_R})=0$.  We next prove that this projection changes the energy only by a lower-order quantity.  Set
\[
        F_R(t):=Q((U_R)_t)=t^2A_R-\frac dD t^\sigma C_R,
        \qquad t>0.
\]
The equation $F_R(t)=0$ is equivalent to
\[
        A_R=\frac dD C_R t^{\sigma-2}.
\]
Since $Q(U_R)<0$, we have
\[
        A_R<\frac dD C_R.
\]
Hence the positive solution is explicitly given by
\[
        t_R=
        \left(\frac{A_R}{(d/D)C_R}\right)^{\!1/(\sigma-2)}
        \in(0,1).
\]
Moreover, direct computation gives
\[
        \frac{A_R}{(d/D)C_R}
        =1+\frac{A_R-(d/D)C_R}{(d/D)C_R}
        =1+\frac{Q(U_R)}{(d/D)C_R}.
\]
By \eqref{eq:Aexp} and \eqref{eq:Cexp} we have $C_R\to K>0$ and
\[
        Q(U_R)=-a_R+\frac dD\gamma_R+o(R^{2-D})=O(R^{2-D}).
\]
Hence 
\[
        \frac{A_R}{(d/D)C_R}=1+O(R^{2-D}).
\]
Since the map $s\mapsto s^{1/(\sigma-2)}$ is $C^1$ in a neighborhood of $s=1$ (recall that $\sigma>2$), we obtain
\begin{equation}
\label{eq:one-minus-t}
        1-t_R=O(R^{2-D}).
\end{equation}
In particular, $t_R\in[1/2,1]$ for all sufficiently large $R$.

It remains to estimate the energy variation along the scaling path.  By \eqref{scale:Ederivative},
\[
        E((U_R)_{t_R})-E(U_R)
        =\int_1^{t_R}s^{-1}F_R(s)\,ds.
\]
For $s\in[t_R,1]$, the derivative
\[
        F_R'(s)=2sA_R-\frac dD\sigma s^{\sigma-1}C_R
\]
is uniformly bounded, because $A_R$ and $C_R$ are bounded and $s\in[1/2,1]$.  Since $F_R(t_R)=0$, the mean value theorem gives
\[
        |F_R(s)|=|F_R(s)-F_R(t_R)|\lesssim |s-t_R|,
        \qquad s\in[t_R,1].
\]
Consequently,
\[
        |E((U_R)_{t_R})-E(U_R)|
        \le C\int_{t_R}^{1}(s-t_R)\,ds
        \lesssim (1-t_R)^2.
\]
Using \eqref{eq:one-minus-t},
\[
        |E((U_R)_{t_R})-E(U_R)|
        \lesssim R^{2(2-D)}=o(R^{2-D}),
\]
because $D\ge3$.  Together with \eqref{eq:energy-before-proj}, this yields
\begin{equation}
\label{eq:proj-strict}
        E((U_R)_{t_R})<\frac{\mathcal S^{D/2}}D
\end{equation}
for all sufficiently large $R$.

\eqref{eq:proj-strict} already gives a strict bound for the energy of some function with vanishing semivirial, but such functions may still have large mass. In the final step we shall use Lemma \ref{lem:monotonicity} to solve this issue. Let
\[
        \mu_R:=M((U_R)_{t_R})=M(U_R).
\]
We claim that $\mu_R\to0$ as $R\to\infty$.  Indeed,
\begin{equation*}
        M(U_R)=\lambda_R^{-2}\int_{\Sigma_R}\eta_R(X)^2W(X,Y)^2\,dX\,dY.
\end{equation*}
If $D\ge5$, then $W\in L^2(\R^D)$ and consequently $M(U_R)\lesssim R^{-2}\to0$.  If $D=4$, then
\[
        W(X,Y)^2\lesssim(1+|X|^2+Y^2)^{-2},
\]
and
\[
        \int_{\R^3}(1+|X|^2+Y^2)^{-2}\,dX
        \lesssim(1+Y^2)^{-1/2}.
\]
Hence
\[
        M(U_R)\lesssim R^{-2}\int_{-R}^R(1+Y^2)^{-1/2}\,dY
        \lesssim R^{-2}\log R\to0.
\]
If $D=3$, the cutoff is essential.  Since $W^2\lesssim(1+|X|^2+Y^2)^{-1}$ and $|X|\le2L_R=2R^2$ on the support of $\eta_R$,
\[
        \int_{|Y|<R,\ |X|<2L_R}W^2\,dX\,dY
        \lesssim R\log R,
\]
so $M(U_R)\lesssim R^{-1}\log R\to0$.  This proves $\mu_R\to0$. Thus $(U_R)_{t_R}\in V(\mu_R)$ and \eqref{eq:proj-strict} implies
\[
        m_{\mu_R}\le E((U_R)_{t_R})<\frac{\mathcal S^{D/2}}D.
\]
Given arbitrary $c>0$, choose $R$ so large that $0<\mu_R<c$.  By Lemma~\ref{lem:monotonicity},
\[
        m_c\le m_{\mu_R}<\frac{\mathcal S^{D/2}}D.
\]
This proves the theorem.
\end{proof}

\section{Proof of the main result}\label{sec:mainproof}
In this final section we give the proof of Theorem~\ref{main theorem}.  The proof follows the compactness scheme of \cite[Prop. 2.7]{Luo_energy_crit}; the new input is Theorem~\ref{theorem strict bound}, which excludes Euclidean critical bubbles at the variational level.

Let
\[
        m_{\mathrm{bub}}:=\frac{\mathcal S^{D/2}}D.
\]
By Theorem~\ref{theorem strict bound},
\begin{equation}
\label{eq:strict-again}
        m_c<m_{\mathrm{bub}}
        \qquad\text{for every }c>0.
\end{equation}

\begin{lemma}[Non-vanishing below the bubble level]\label{lem:nonvanishing}
Let $c>0$, and let $(u_n)_n\subset V(c)$ be a minimizing sequence for $m_c$.  Then there exist translations $x_n\in\R^d$ and a nonzero $u\in H^1(\mathbb X)$ such that, after passing to a subsequence,
\[
        u_n(\cdot+x_n,\cdot)\rightharpoonup u
        \qquad\text{weakly in }H^1(\mathbb X).
\]
\end{lemma}
\begin{proof}
Since $Q(u_n)=0$,
\[
        \|\nabla_xu_n\|_2^2=\frac dD\|u_n\|_{2^*}^{2^*}.
\]
Therefore
\begin{align*}
        E(u_n)
        &=\frac12\|\partial_yu_n\|_2^2
          +\frac1{2d}\|\nabla_xu_n\|_2^2.
\end{align*}
As $E(u_n)\to m_c<\infty$ and $M(u_n)=c$, the sequence is bounded in $H^1(\mathbb X)$.

We also claim that the $L^{2^*}$ norm does not vanish.  By Lemma \ref{lemma gn additive} and $Q(u_n)=0$,
\[
        \|\nabla_xu_n\|_2^2
        =\frac dD\|u_n\|_{2^*}^{2^*}
        \le C\|\nabla_xu_n\|_2^{\frac{2d}{d-1}}
        \left(c^{\frac1{d-1}}+\|\partial_yu_n\|_2^{\frac2{d-1}}\right).
\]
The factor in parentheses is uniformly bounded.  Since $u_n\ne0$, $\|\nabla_xu_n\|_2>0$.  Dividing by $\|\nabla_xu_n\|_2^2$ gives
\[
        1\le C\|\nabla_xu_n\|_2^{\frac2{d-1}},
\]
after adjusting $C$.  Hence
\begin{equation}
\label{eq:nonzero-grad}
        \liminf_{n\to\infty}\|\nabla_xu_n\|_2^2>0,
        \qquad
        \liminf_{n\to\infty}\|u_n\|_{2^*}^{2^*}>0.
\end{equation}

Assume by contradiction that every sequence of $x$-translations has weak limit zero.  Apply now Lemma~\ref{lem:profile}. We show first that no nonzero scale-one frames can occur.  Indeed, if such a profile existed, then for some frame \(\mathcal F^j=(1,(x_n^j,y_n^j))\) the corresponding term would be \(\Pi_{(x_n^j,y_n^j)}\phi^j\) with \(0\ne\phi^j\in H^1(\mathbb X)\).  Since \(\T\) is compact, after passing to a subsequence we may assume \(y_n^j\to y_\infty\).  Translating \(u_n\) by \(x_n^j\) in the Euclidean variables and using the orthogonality of all other frames, we would obtain a nonzero weak limit, namely a translate of \(\phi^j\) in the \(y\)-variable.  This contradicts the assumed weak vanishing, and hence all nonzero profiles are Euclidean.  By \eqref{eq:nonzero-grad}, there exists at least one  nonzero Euclidean profile. W.l.o.g. we may assume in the following that all profiles $\phi^j$ are nonzero and Euclidean for any $j<J_*$.

For a Euclidean profile $\phi^j\in\dot H^1(\R^D)$ write, with $Z=(X,Y)$,
\[
        A_j:=\|\nabla_X\phi^j\|_{L^2(\R^D)}^2,
        \quad
        B_j:=\|\partial_Y\phi^j\|_{L^2(\R^D)}^2,
        \quad
        C_j:=\|\phi^j\|_{L^{2^*}(\R^D)}^{2^*}.
\]
Set
\[
        q_j:=A_j-\frac dD C_j,
        \qquad
        i_j:=\frac12B_j+\frac1{2D}C_j.
\]
From the decompositions \eqref{eq:prof-x}--\eqref{eq:prof-lp} and $Q(u_n)=0$, for finite $J$,
\begin{align}
        0=\sum_{j=1}^Jq_j
        +\left(\|\nabla_xr_n^J\|_2^2-\frac dD\|r_n^J\|_{2^*}^{2^*}\right)+o_n(1).\label{new}
\end{align}
Since \eqref{eq:prof-rem} makes $\|r_n^J\|_{2^*}^{2^*}$ arbitrarily small as $J\to J_*$, at least one nonzero Euclidean profile satisfies
\begin{equation}
\label{eq:qj-negative}
        A_j-\frac dD C_j\le0.
\end{equation}
Indeed, if all nonzero profiles had $q_j>0$, choosing one fixed nonzero profile and then taking $J$ large enough so that $\|r_n^J\|_{2^*}^{2^*}$ is negligible would contradict \eqref{new}.

From now on we fix some $j$ so that \eqref{eq:qj-negative} is fulfilled. We may simply drop the index $j$ and write $A,B,C,i$.  Thus
\begin{equation}
\label{eq:A-le-C}
        A\le\frac dD C.
\end{equation}
For $\alpha,\beta>0$ define
\[
        \phi_{\alpha,\beta}(X,Y)
        :=(\alpha^d\beta)^{\frac{D-2}{2D}}\phi(\alpha X,\beta Y).
\]
Then $\|\phi_{\alpha,\beta}\|_{2^*}=\|\phi\|_{2^*}$ and, with $r:=\alpha/\beta$,
\[
        \|\nabla_X\phi_{\alpha,\beta}\|_2^2=r^{2/D}A,
        \qquad
        \|\partial_Y\phi_{\alpha,\beta}\|_2^2=r^{-2d/D}B.
\]
Applying the sharp Sobolev inequality to $\phi_{\alpha,\beta}$ and optimizing over $r>0$ yields
\begin{equation}
\label{eq:anisotropic-sob}
        \mathcal S C^{2/2^*}
        \le \inf_{r>0}\left(r^{2/D}A+r^{-2d/D}B\right)
        =D d^{-d/D}A^{d/D}B^{1/D}.
\end{equation}
Raising \eqref{eq:anisotropic-sob} to the power $D$ and using \eqref{eq:A-le-C},
\[
        \mathcal S^D C^{D-2}
        \le D^D d^{-d}A^dB
        \le D^D d^{-d}\left(\frac dD C\right)^dB
        =D C^dB.
\]
Since $D=d+1$, this gives
\begin{equation*}
        BC\ge \frac{\mathcal S^D}{D}.
\end{equation*}
Therefore
\begin{equation*}
        i=\frac12B+\frac1{2D}C
        \ge \sqrt{\frac{BC}{D}}
        \ge \frac{\mathcal S^{D/2}}D
        =m_{\mathrm{bub}}.
\end{equation*}
On the other hand, using $I=E$ on $Q=0$ and the profile decomposition for $I$,
\[
        m_c=\lim_{n\to\infty}I(u_n)
        \ge i_j\ge m_{\mathrm{bub}},
\]
contradicting \eqref{eq:strict-again}.  The contradiction proves the lemma.
\end{proof}

With help of Lemma \ref{lem:nonvanishing} we are finally able to give the proof for Theorem~\ref{main theorem}.

\begin{proof}[Proof of Theorem~\ref{main theorem}]
This is essentially the same proof as the one for \cite[Prop. 2.7]{Luo_energy_crit}, for the sake of completeness we give the full details here. Fix $c>0$ and let $(u_n)_n\subset V(c)$ be a minimizing sequence.  By Lemma~\ref{lem:nonvanishing}, after translating in $x$ and passing to a subsequence we may assume
\[
        u_n\rightharpoonup u\ne0
        \qquad\text{weakly in }H^1(\mathbb X).
\]
By replacing $u_n$ with $|u_n|$, and using the diamagnetic inequality and Lemma \ref{lem:lecoz}, we may suppose $u_n\ge0$ and hence $u\ge0$.

Let
\[
        c_1:=M(u)\in(0,c].
\]
Set $v_n:=u_n-u$.  By the Hilbert space decomposition and the Brezis--Lieb lemma \cite{BrezisLieb1983},
\begin{align}
        M(u_n)&=M(u)+M(v_n)+o(1),\notag\\
        Q(u_n)&=Q(u)+Q(v_n)+o(1),\label{eq:BL-Q}\\
        I(u_n)&=I(u)+I(v_n)+o(1).\notag
\end{align}
We first show that $Q(u)\le0$.  Suppose $Q(u)>0$.  Since $Q(u_n)=0$, \eqref{eq:BL-Q} implies $Q(v_n)<0$ for large $n$.  By Lemma~\ref{lem:projection}, there exists $t_n\in(0,1)$ such that $Q((v_n)_{t_n})=0$ and
\[
        I((v_n)_{t_n})<I(v_n).
\]
Moreover $M(v_n)=c-c_1+o(1)<c$ for large $n$.  Lemmas~\ref{lem:lecoz} and \ref{lem:monotonicity} then imply
\begin{align*}
        m_c
        &\le m_{M(v_n)}
        \le I((v_n)_{t_n})
        <I(v_n) \\
        &=I(u_n)-I(u)+o(1)
        =m_c-I(u)+o(1).
\end{align*}
Letting $n\to\infty$ it follows $I(u)\le0$.  Since
\[
        I(u)=\frac12\|\partial_yu\|_2^2+\frac1{2D}\|u\|_{2^*}^{2^*}
\]
and $u\ne0$, we get a contradiction, and consequently $Q(u)\le0$. We next show that $Q(u)$ can also not be negative.  If $Q(u)<0$, Lemma~\ref{lem:projection} yields some $s\in(0,1)$ with $Q(u_s)=0$ and $I(u_s)<I(u)$.  Since $M(u_s)=c_1$, by Lemmas~\ref{lem:lecoz} and \ref{lem:monotonicity},
\[
        m_{c_1}\le I(u_s)<I(u)
        \le \liminf_{n\to\infty}I(u_n)=m_c\le m_{c_1},
\]
a contradiction.  Therefore
\begin{equation*}
        Q(u)=0.
\end{equation*}
Thus $u\in V(c_1)$ and
\[
        m_{c_1}\le E(u)=I(u)\le m_c\le m_{c_1}.
\]
Consequently
\begin{equation}
\label{eq:u-minimizer-c1}
        E(u)=m_{c_1}=m_c,
\end{equation}
and $u$ is an optimizer for $m_{c_1}$. 

By Lemma \ref{minimizer is solution}, $u$ solves the Euler-Lagrange equation
\begin{equation}
\label{eq:EL-final}
        -\Delta_{x,y}u+\beta u=u^{2^*-1}.
\end{equation}
We prove that $\beta>0$.  Testing \eqref{eq:EL-final} against $u$ gives
\[
        \|\nabla_xu\|_2^2+\|\partial_yu\|_2^2+\beta M(u)=\|u\|_{2^*}^{2^*}.
\]
Using $Q(u)=0$, i.e. $\|\nabla_xu\|_2^2=\frac dD\|u\|_{2^*}^{2^*}$, we obtain
\begin{equation}
\label{eq:beta-identity}
        \|\partial_yu\|_2^2+\beta M(u)=\frac1D\|u\|_{2^*}^{2^*}.
\end{equation}
Now introduce the mass-changing scaling
\begin{equation*}
        T_\lambda u(x,y):=\lambda^{\frac{d-1}{2}}u(\lambda x,y),
        \qquad \lambda>0.
\end{equation*}
Then
\begin{align}
        M(T_\lambda u)&=\lambda^{-1}M(u),\nonumber\\
        Q(T_\lambda u)&=\lambda Q(u)=0,\nonumber\\
        E(T_\lambda u)&=\frac\lambda{2D}\|u\|_{2^*}^{2^*}+\frac{\lambda^{-1}}2\|\partial_yu\|_2^2.\label{new2}
\end{align}
For $\lambda>1$, $M(T_\lambda u)=c_1/\lambda<c_1$.  Since $u$ minimizes $m_{c_1}$ and $c\mapsto m_c$ is nonincreasing,
\[
        E(T_\lambda u)\ge m_{c_1/\lambda}\ge m_{c_1}=E(u).
\]
Taking the right derivative at $\lambda=1$ yields
\begin{equation}
\label{eq:beta-nonneg}
        \frac1D\|u\|_{2^*}^{2^*}-\|\partial_yu\|_2^2\ge0.
\end{equation}
Combining \eqref{eq:beta-identity} and \eqref{eq:beta-nonneg} gives $\beta M(u)\ge0$.  Since $M(u)=c_1>0$, $\beta\ge0$.

It remains to exclude $\beta=0$.  If $\beta=0$, then
\begin{equation}
\label{eq:critical-beta-zero}
        -\Delta_{x,y}u=u^{\frac{D+2}{D-2}}.
\end{equation}
By the Brezis--Kato estimate and local elliptic regularity, $u\in C^2(\R^d\times\T)$; see \cite{BrezisKato,Struwe1996}.  The strong maximum principle gives $u>0$.  Lifting $u$ periodically along the $y$-direction to $\R^D$, we obtain a positive $C^2$ solution of \eqref{eq:critical-beta-zero} on $\R^D$.  The Caffarelli--Gidas--Spruck classification theorem \cite{Liouville3} implies that such a positive solution must be an Aubin--Talenti bubble
\[
        u(Z)=a\left(1+b|Z-Z_0|^2\right)^{-\frac{D-2}{2}}
\]
with $a,b>0$.  Such a function is not periodic in the $Y$-variable unless it is identically zero, contradicting $u>0$.  Hence $\beta>0$.

Finally, we show that no mass is lost, namely $c_1=c$.  Suppose by contradiction that $c_1<c$.  From \eqref{eq:u-minimizer-c1}, $m_{c_1}=m_c$.  Since $c\mapsto m_c$ is nonincreasing, it follows that $m_\rho=m_{c_1}$ for every $\rho\in[c_1,c]$.  If $0<\lambda<1$ is sufficiently close to $1$, then $M(T_\lambda u)=c_1/\lambda\in[c_1,c]$, and hence
\[
        E(T_\lambda u)\ge m_{c_1/\lambda}=m_{c_1}=E(u).
\]
For $\lambda>1$ sufficiently close to $1$, the inequality $E(T_\lambda u)\ge E(u)$ was already obtained from monotonicity, because $m_{c_1/\lambda}\ge m_{c_1}$.  Therefore $\lambda=1$ is a genuine local minimizer of the differentiable function $\lambda\mapsto E(T_\lambda u)$, and consequently
\[
        \left.\frac d{d\lambda}E(T_\lambda u)\right|_{\lambda=1}=0.
\]
Using the explicit formula for $E(T_\lambda u)$ given in \eqref{new2}, this yields
\begin{equation}
\label{eq:equality-derivative}
        \|\partial_yu\|_2^2=\frac1D\|u\|_{2^*}^{2^*}.
\end{equation}
Combining \eqref{eq:equality-derivative} with \eqref{eq:beta-identity}, we obtain
\[
        \beta M(u)=0,
\]
contradicting $\beta>0$ and $M(u)>0$.  Therefore $c_1=c$.

Thus $u\in V(c)$ and $E(u)=m_c$.  Since $u\ge0$ solves \eqref{eq:EL-final} with $\beta>0$, the strong maximum principle gives $u>0$.  This completes the proof of Theorem~\ref{main theorem}.
\end{proof}

\subsubsection*{Acknowledgements}
Y. Luo was supported by the NSF grant of Guangdong (No. 2024A1515010497), the QB-Program of Guangdong (No. 2024QN11X141) and the NSF grant of China (No. 12301301).

\subsubsection*{Data availability}
Data sharing not applicable to this article as no datasets were generated or analysed during the current study.


\end{document}